\documentclass[11pt,oneside,reqno]{amsart}
\usepackage{amssymb, url, color, pb-diagram, graphicx, amscd, pb-diagram,mathrsfs,amsfonts}
\usepackage[colorlinks=true, bookmarks=true, pdfstartview=FitH, pagebackref=true]{hyperref}
\usepackage[nodayofweek]{datetime}
\usepackage{enumerate,stmaryrd}
\usepackage[utf8]{inputenc}
\usepackage[english]{babel}
\usepackage{amstext, amsmath,latexsym,amsbsy,amssymb,amsthm,graphicx}
\usepackage{mathrsfs}
\usepackage{txfonts}
\usepackage{amsfonts} 
\usepackage{euscript}
\usepackage{verbatim}
\usepackage{graphicx}
\usepackage{times}
\usepackage{setspace}
\usepackage[space, compress, sort]{cite}
\usepackage{empheq}
\newtheorem{Lemma}{Lemma}[section]

\newtheorem{Remark}[Lemma]{Remark}
\newtheorem{Proposition}[Lemma]{Proposition}
\newtheorem{Theorem}[Lemma]{Theorem} 
\newtheorem{Corollary}[Lemma]{Corollary}

\newcounter{mnotecount}[section]

\begin{document}
\title[Nonexistence and non uniqueness for the constraint equations]{ Nonexistence and Nonuniqueness Results for solutions to the Vacuum Einstein Conformal Constraint Equations}
\author[T.C Nguyen]{The-Cang Nguyen}
\date{October 25, 2015}

\address{Laboratoire de Mathématiques et Physique Théorique Université de Tours\\
UFR Sciences et Techniques\\
Parc de Grandmont\\
37200 Tours - FRANCE}
\email{The-Cang.Nguyen@lmpt.univ-tours.fr}


\begin{abstract}
In this article, we give nonexistence and nonuniqueness results for
the vacuum Einstein conformal constraint equations in the far from
CMC case and also show that in some cases the equations of the
conformal method for positive Yamabe metrics and with TT-tensor
$\sigma \equiv 0$ have a non-trivial solution, and thus answer a
question by D. Maxwell \cite{MaxwellNonCMC}.
\end{abstract}
\maketitle
\section{Introduction}

\subsection{Background}

In general relativity, a space-time is a $(n+1)-$dimensional Lorentzian
manifold $(\mathcal{M},h)$ (i.e, $h$ has signature $-~+~+~...~+$),
with $n\geq 3$ which satisfies The Einstein equations

\begin{equation}\label{Eins.eq}
\mathrm{Ric}^h_{\mu\nu}-\frac{R_{h}}{2}h_{\mu\nu}=\frac{8\pi\mathcal{G}}{c^{4}}T_{\mu\nu},
\end{equation}
where $\mathrm{Ric}^h$ and $R_{h}$ are respectively the Ricci and the
scalar curvatures of of $h$, $\mathcal{G}$ is Newton's constant, $c$
is the speed of light and $T$ is the stress-energy tensor of
non-gravitational fields (i.e. matter fields, electromagnetic field...).\\

Einstein equations are roughly speaking hyperbolic of order 2. Hence
all solutions can be obtained from their initial values at some
``time t=0'', the metric $\hat{g}$ induced on a Cauchy hypersurface
$M\subset \mathcal{M}$, and its initial velocity, the second fundamental
form $\hat{K}$ of the embedding $M\subset \mathcal{M}$. By the Gauss and
Codazzi equations, the choice of $(M,\hat{g},\hat{K})$ from \eqref{Eins.eq}
must satisfy the so-called Einstein constraint equations. In the vacuum case,
i.e. when $T\equiv 0$, these equations are

\begin{equation}\label{vacuum.ECE}
\begin{aligned}
R_{\hat{g}}-|\hat{K}|_{\hat{g}}^{2}+\left(\text{tr}_{\hat{g}}\hat{K}\right)^{2} &=0,\\
\hat{K}-d_{\hat{g}}~\text{tr}_{\hat{g}}\hat{K} &=0.
\end{aligned}
\end{equation}

Constructing and classifying solutions of this system is an important issue.
For a deeper discussion of \eqref{vacuum.ECE} , we refer the reader to the
excellent review article \cite{BartnikIsenberg}. One of most efficient
methods to find initial data satisfying \eqref{vacuum.ECE} is the conformal
method developed by Lichnerowicz \cite{Li44} and Y. Choquet-Bruhat-Jr. York
\cite{CBY80}. The idea of this method is to effectively parameterize the
solutions to \eqref{vacuum.ECE} by some reasonable parts and then solve for
the rest of the data. More precisely, we assume given some initial data: a
Riemannian manifold $(M, g)$ which we will assume compact, a mean curvature
$\tau$ (a function), a transverse-traceless tensor $\sigma$ (i.e. a symmetric,
trace-free, divergence-free $(0,2)$-tensor). Then we look for a positive function
$\varphi$ and a $1-$form $W$ such that
\[
\hat{g} =\varphi^{N-2}g,\quad
\hat{K}=\frac{\tau}{n}\varphi^{N-2}g + \varphi^{-2} (\sigma+LW)
\]
is a solution to the vacuum Einstein constraint equations \eqref{vacuum.ECE}.
Here $N=\frac{2n}{n-2}$ and $L$ is the conformal Killing operator defined by
\[
 LW_{ij}=\nabla_{i} W_{j}+\nabla_{j}W_{i}-\frac{2}{n} \nabla^k W_k g_{ij},
\]
where $\nabla$ is the Levi-Civita connection associated to the metric $g$. \\

Equations \eqref{vacuum.ECE} can be reformulated in terms of $\varphi$ and $W$
as follows:
\begin{subequations}\label{CE}
\begin{eqnarray}
\frac{4(n-1)}{n-2}\Delta_{g} \varphi + R_{g} \varphi &=& -\frac{n-1}{n}\tau^{2} \varphi^{N-1} + |\sigma+LW|_{g}^{2}\varphi^{-N-1} \quad [\text{Lichnerowicz equation}],\\
-\frac{1}{2}L^{*}LW &=& \frac{n-1}{n}\varphi^{N} d\tau\qquad\qquad\quad ~~~~~~~~~~~~~~~~~~~~[\text{vector equation}],
\end{eqnarray}
\end{subequations}
where $\Delta_{g}$ is the nonnegative Laplace operator and $L^{*}$ is the formal
$L^{2}-$adjoint of $L$.

These coupled equations are called \textit{the conformal constraint equations}.
During the past decades, many existence and uniqueness results for \eqref{CE}
were proven. They depend on the Yamabe invariant $\mathcal{Y}_{g}$ of the metric $g$
defined by
$$\mathcal{Y}_{g}=\inf_{\substack{f\in C^{\infty}(M)\\f\nequiv 0}}\frac{\frac{4(n-1)}{n-2}\int_{M}{|\nabla f|^{2}dv}+\int_{M}{Rf^{2}}}{||f||^{2}_{L^{N}(M)}}.$$

When $\tau$ is constant, the system \eqref{CE} becomes uncoupled (since $d\tau\equiv 0$
in the vector equation) and a complete description of the situation was achieved by
J. Isenberg \cite{Isenberg}. The near CMC case (i.e. when $d\tau$ is small) was
addressed soon after. Most results can be found in \cite{BartnikIsenberg}. For
arbitrary $\tau$ however, the situation appears much harder and only two methods
exist to tackle this case. The first one, obtained by Holst-Nagy-Tsogtgerel \cite{HNT2}
and Maxwell \cite{MaxwellNonCMC}, shows that the system \eqref{CE} admits a solution,
provided $g$ has positive Yamabe invariant and $\sigma\nequiv 0$ is small enough.
The second one, introduced by Dahl-Gicquaud-Humbert \cite{DahlGicquaudHumbert},
states that if $\tau$ has constant sign and if \textit{the limit equation}
\begin{equation}\label{limit.eq}
 -\frac{1}{2}L^{*}LV=\alpha\sqrt{\frac{n-1}{n}}|LV|\frac{d\tau}{\tau}
\end{equation}
has no non-zero solution $V$, for all values of the parameter $\alpha \in [0, 1]$,
then the set of solutions $(\varphi, W)$ to \eqref{CE} is not empty and compact.
This criterion holds true e.g. when $(M, g)$ has $\mathrm{Ric} \leq -(n-1)g$,
with $\left\|\frac{d\tau}{\tau}\right\|_{L^\infty} < \sqrt{n}$
(see also \cite{GicquaudSakovich} for an extension of this result to asymptotically
hyperbolic manifolds). An unifying point of view of these results is given in
\cite{GicquaudNgo} and \cite{Nguyen}.\\

Conversely, nonexistence and nonuniqueness results for \eqref{CE} are fairly rare.
We refer to arguments of Rendall, as presented in \cite{IsenbergOMurchadha},
Holst-Meier \cite{HolstMeierNonuniqueness}, and Dahl-Gicquaud-Humbert
\cite{DahlGicquaudHumbertNonExistence} for attempts to obtain such results.
In the vacuum case, the only model of nonuniqueness of solutions is constructed on
the $n-$torus by D. Maxwell \cite{MaxwellConformalParameterization} while the only
nonexistence result, achieved by J. Isenberg-Murchadha \cite{IsenbergOMurchadha} and
later strengthened in \cite{DahlGicquaudHumbert} and \cite{GicquaudNgo}, states that
the system \eqref{CE} with $\sigma\equiv 0$ has no solution when $\mathcal{Y}_{g}\geq 0$
and  $d\tau/\tau$ is small enough. This assertion together with experimentations on the
torus led D. Maxwell to post a question concerning whether the non-zero assumption of $\sigma$ is a necessary
condition for existence of solution to the conformal equations \eqref{CE} with positive
Yamabe invariant (see \cite{MaxwellConformalParameterization}).\\

In this article, based on an idea from \cite{GicquaudNgo}, we give another version of
the main theorem in \cite{DahlGicquaudHumbert} and \cite{Nguyen}, which allows $\alpha$
in the limit equation \eqref{limit.eq} to be set to $1$. Next we give seed data in the
far from CMC case for which the system \eqref{CE} has no solution. As a direct
consequence of this result, we exhibit cases of nonuniqueness of solutions and
give an answer to D. Maxwell's question stated above.

\subsection{Statement of results;}
 Let $M$ be a compact manifold of dimension $n\geq 3$. Our goal is to study solutions to the vacuum Einstein equations using the conformal method. The given data on $M$ consist in
 \begin{equation}\label{condintial}
 \begin{aligned}
 &\bullet~~\mbox{a Riemannian metric $g\in C^{2},$}\\
 &\bullet~~\mbox{a function $\tau\in W^{1,p}$,}~~~~~~~~~~~~~~~~~~~~~~~~~~~~~~~~~~~~~~~~~~~~~~~~~~~~~~~~~~~~~~~~~~~~~~~~~~~~~~~~~~~~~~~~~~~~~~~~~~~~~~~~~~\\
 &\bullet~~\mbox{a symmetric, trace- and divergence-free $(0,2)-$tensor $\sigma\in W^{1,p}$,}
 \end{aligned}
 \end{equation}
 with $p>n$. One is required to find
 \begin{equation*}
 \begin{aligned}
 &\bullet~~\mbox{a positive function $\varphi\in W^{2,p}$,}~~~~~~~~~~~~~~~~~~~~~~~~~~~~~~~~~~~~~~~~~~~~~~~~~~~~~~~~~~~~~~~~~~~~~~~~~~~~~~~~~~~~~~~~~~~~~~\\
 &\bullet~~\mbox{a $1-$form $W\in W^{2,p}$,}
 \end{aligned}
 \end{equation*}
 which satisfy the conformal constraint equations \eqref{CE}. We also assume that
 \begin{equation}\label{condinitial2}
 \begin{aligned}
 &\bullet~~\mbox{$\tau^{2}>0$,}\\
 &\bullet~~\mbox{$(M,g)$ has no conformal Killing vector field,}~~~~~~~~~~~~~~~~~~~~~~~~~~~~~~~~~~~~~~~~~~~~~~~~~~~~~~~~~~~~~~~~~~\\
 &\bullet~~\mbox{$\sigma\nequiv 0$}.
 \end{aligned}
 \end{equation}
 
   We use standard notations for function spaces, such as $L^{p}$, $C^{k}$, and Sobolev spaces $W^{k,p}$. It will be clear from the context if the notation refers to a space of functions on $M$, or a space of sections of some bundle over $M$. For spaces of functions which embed into $L^{\infty}$, the subscript $+$ is used to indicate the cone of positive functions.\\
  \\
  We will sometimes write, for instance, $C(\alpha_{1},\alpha_{2})$ to indicate that a constant $C$ depends only on $\alpha_{1}$ and $\alpha_{2}$.\\
  
  After briefly sketching basic facts on the conformal constraint equations \eqref{CE}, in Section 3 we use the Leray-Schauder fixed point theorem introduced in \cite{Nguyen} to obtain the main result of this article, which is another version of \cite[Theorem 1.1]{DahlGicquaudHumbert} and of \cite[Theorem 3.3]{Nguyen}:
 \begin{Theorem}\label{theoremofDHG}
  Let data be given on $M$ as specified in \eqref{condintial} and assume that conditions \eqref{condinitial2} hold. Then at least one of the following assertions is true
 \begin{itemize}
 \item[(i)] The conformal constraint equations \eqref{CE} admit a solution $(\varphi,W)$ with $\varphi>0$. Furthermore, the set of solutions $(\varphi,W)\in W_{+}^{2,p}\times W^{2,p}$ is compact.
 \item[(ii)] There exists a nontrivial solution $V\in W^{2,p}$ to the limit equation
 \begin{equation}\label{limit_non_parameter}
 -\frac{1}{2}L^{*}LV=\sqrt{\frac{n-1}{n}}|LV|\frac{d\tau}{\tau}.
 \end{equation}
 \item[(iii)] For any continuous function $f>0$ or $f\equiv R$ if $\mathcal{Y}_{g}>0$, the (modified) conformal constraint equations
 \begin{subequations}\label{modified_CE}
 \begin{eqnarray}
 \frac{4(n-1)}{n-2}\Delta \varphi+f\varphi&=&-\frac{n-1}{n}\tau^{2}\varphi^{N-1}+|LW|^{2}\varphi^{-N-1}\\
 -\frac{1}{2}L^{*}LW&=&\frac{n-1}{n}\varphi^{N}d\tau
 \end{eqnarray}
 \end{subequations}
 have a (non-trivial) solution $(\varphi,W)\in W^{2,p}_{+}\times W^{2,p}$. Moreover if $\mathcal{Y}_{g}>0$, there exists a sequence $\{t_{i}\}$ converging to $0$ s.t. the conformal constraint equations \eqref{CE} associated to the seed data $(g,t_{i}\tau,\sigma)$ have at least two solutions. 
 \end{itemize}
 \end{Theorem}
 Comparing with the original version of Dahl-Gicquaud-Humbert, the price to pay to control the parameter ($\alpha=1$) is the addition of $(iii)$. However, we will see that this assertion is necessary (see Theorem \ref{Theorem.nonexistence} below).\\

 In Section 4 we present several applications of Theorem \ref{theoremofDHG}. The basic idea of these applications is to seek seed data such that neither $(i)$ nor $(ii)$ in Theorem \ref{theoremofDHG} holds. It follows then that $(iii)$ is satisfied. In this approach, one of our main result is the following:
 \begin{Theorem}\label{Theorem.nonexistence}(\textbf{\textit{Nonexistence of solution}})
 Let data be given on $M$ as specified in \eqref{condintial} and assume that conditions \eqref{condinitial2} hold. Furthermore, assume that there exists $c>0$ s.t. $\left|L\left(\frac{d\tau}{\tau}\right)\right|\leq c\left|\frac{d\tau}{\tau}\right|^{2}$. Let $V$ be a given open neighborhood of the critical set of $\tau$. If $\sigma\nequiv 0$ and $\text{supp}\{\sigma\}\subsetneq M\setminus V$, then both of the conformal constraint equations \eqref{CE} and the limit equation \eqref{limit_non_parameter} associated to the seed data $(g,\tau^{a},k\sigma)$ have no (nontrivial) solution, provided $a$ and $k$ are large enough.
 \end{Theorem}
 We point out that \cite[Proposition 1.6]{DahlGicquaudHumbert} provides the existence of seed data satisfying such assumptions. In fact, our proof for Theorem \ref{Theorem.nonexistence} is an extension of arguments in \cite[Proposition 1.6]{DahlGicquaudHumbert}. It is worth noting that 
 $$\left|\frac{d\tau^{a}}{\tau^{a}}\right|=a\left|\frac{d\tau}{\tau}\right|.$$
 Therefore given non-constant $|\tau|>0$, provided that $a$ is large enough, $\tau^{a}$ is a far-from-CMC. Moreover, as we will see later in the proof, the role of $(a,k)$ in Theorem \ref{Theorem.nonexistence} is as follows. We need the largeness assumption of $a$ to ensure that the limit equation \eqref{limit_non_parameter} associated to $(g,\tau^{a})$ has no solution. It follows that given $a$ large enough depending on $(g,\tau,c)$,  the set
 $$S_{a}=\Bigl\{(\varphi,W)~|~ \mbox{$\exists k\in\mathrm{R}_{+}$~:~ $(\varphi,W)$ is a solution to \eqref{CE} associated to $(g,\tau^{a},k\sigma)$} \Bigr\}$$
 is bounded in $W^{2,p}_{+}\times W^{2,p}$.  That means that the system \eqref{CE} associated to $(g,\tau^{a},k\sigma)$ has no solution for all $k$ large enough depending on $(g,\tau,\sigma,a)$ as claimed.   \\

 As direct consequences of Theorem \ref{theoremofDHG} and \ref{Theorem.nonexistence}, we also obtain the following results.
 \begin{Corollary}\label{Maxwell'squestion}(\textbf{\textit{An answer to Maxwell's question}})
Let $(M,g,\tau)$ be given as in Theorem  \ref{Theorem.nonexistence}. If $\mathcal{Y}_{g}>0$, then the conformal constraint equations \eqref{CE} associated to $(g,\tau^{a},0)$ have a (nontrivial)  solution for all $a>0$ large enough. 
 \end{Corollary}
 \begin{Corollary}\label{Theorem.nonunique}(\textbf{\textit{Nonuniqueness of solutions}})
 Assume that $(M,g,\tau,\sigma,a,k)$ is given as in Theorem \ref{Theorem.nonexistence}. If $\mathcal{Y}_{g}>0$, then there exists a sequence $\{t_{i}\}$ converging to $0$ s.t. the conformal constraint equations \eqref{CE} associated to seed data $(g,t_{i}\tau^{a},k\sigma)$ have at least two solutions.
 \end{Corollary} 
 \subsection*{Acknowledgements}
  \addcontentsline{toc}{section}{Acknowledgements}
  The author wishes to express his gratitude to Romain Gicquaud for his help in proving Theorem \ref{theoremofDHG} and his great patience and care in the proofreading of preliminary versions of this article. The author would also like to thank Emmanuel Humbert for his advice and helpful discussions.
  
  \section{Preliminaries}
  In this section, we review some standard facts about the Lichnerowicz equation on a compact $n-$manifold $M$:
  \begin{equation}\label{Lichnerowicz}
  \frac{4(n-1)}{n-2}\Delta u+Ru+\frac{n-1}{n}\tau^{2}u^{N-1}=\frac{w^{2}}{u^{N+1}}.
  \end{equation}
  Given a function $w$ and $p>n$, we say that $u_{+}\in W_{+}^{2,p}$ is a \textit{supersolution} to \eqref{Lichnerowicz} if
  $$\frac{4(n-1)}{n-2}\Delta u_{+}+Ru_{+}+\frac{n-1}{n}\tau^{2}u_{+}^{N-1}\geq\frac{w^{2}}{u_{+}^{N+1}}.$$
  A \textit{subsolution} is defined similarly with the reverse inequality.
  \begin{Proposition} \label{method sub-super}(see \cite{MaxwellRoughCompact})
  Assume $g\in C^{2}$ and $w,\tau\in L^{2p}$ for some $p>n$. If $u_{-},u_{+}\in W_{+}^{2,p}$ are respectively a subsolution and a supersolution  to \eqref{Lichnerowicz} associated with a fixed $w$ such that $u_{-}\leq u_{+}$, then there exists a solution $u\in W_{+}^{2,p}$ to \eqref{Lichnerowicz} such that $u_{-}\leq u\leq u_{+}$.
  \end{Proposition}
  \begin{Theorem}\label{maxwell1}(see \cite{Isenberg} and \cite{MaxwellRoughCompact})
  Assume $w,\tau\in L^{2p}$ and $g\in C^{2}$ for some $p>n$. Then there exists a positive solution $u\in W^{2,p}_{+}$ to \eqref{Lichnerowicz} if and only if one of the following assertions is true.
  \begin{itemize}
  \item[1.] $\mathcal{Y}_{g}>0$ and $w\nequiv 0$,
  \item[2.] $\mathcal{Y}_{g}=0$ and $w\nequiv 0$, $\tau\nequiv 0$,
  \item[3.] $\mathcal{Y}_{g}<0$ and there exists $\hat{g}$ in the conformal class of $g$ such that $R_{\hat{g}}=-\frac{n-1}{n}\tau^{2}$,
  \item[4.]  $\mathcal{Y}_{g}=0$ and $w\equiv 0$, $\tau\equiv 0.$
  \end{itemize}
  In Cases $1-3$ the solution is unique. In Case $4$ any two solutions are related by a scaling by a positive constant multiple. Moreover, Case $3$ holds if $\mathcal{Y}_{g}<0$ and the set of zero-points of $\tau$ has zero Lebesgue measure (see \cite{Rauzy} or \cite[Theorem 6.12]{Aubin}). In particular, existence and uniqueness are guaranteed if $|\tau|>0$ and $w\nequiv 0$ independently of $\mathcal{Y}_{g}$.
  \end{Theorem} 
  The main technique used to prove the theorem above is the conformal covariance of \eqref{Lichnerowicz}.
  \begin{Lemma}\label{maxwell2} (see \cite[Lemma 1]{MaxwellNonCMC})
  Assume $g\in C^{2}$ and $w,\tau\in L^{2p}$ for some $p>n$. Assume also that $\phi\in W^{2,p}_{+}$. Define
  $$\hat{g}=\phi^{\frac{4}{n-2}}g,~~~\hat{w}=\phi^{-N}w,~~~\hat{\tau}=\tau.$$
  Then $u$ is a supersolution (resp. subsolution) to \eqref{Lichnerowicz} if and only if $\hat{u}=\phi^{-1}u$ is a supersolution (resp. subsolution) to the conformally transformed equation
  \begin{equation}\label{Lichnerowicz2}
  \frac{4(n-1)}{n-2}\Delta_{\hat{g}} \hat{u}+R_{\hat{g}}\hat{u}+\frac{n-1}{n}\hat{\tau}^{2}\hat{u}^{N-1}=\frac{\hat{w}^{2}}{\hat{u}^{N+1}}.
  \end{equation}
  In particular, $u$ is a solution to \eqref{Lichnerowicz} if and only if $\hat{u}$ is a solution to \eqref{Lichnerowicz2}.
  \end{Lemma}
   From the techniques in \cite{GicquaudNgo}, we get the following remark.
   \begin{Remark}\label{changeR}
   	Theorem \ref{maxwell1} guarantees that given any $w\in C^{0}\setminus \{0\}$, there exists a unique corresponding solution $u\in W^{2,p}_{+}$ to \eqref{Lichnerowicz}. 
   	Let $(\hat{g},\hat{w},\hat{\tau},\hat{u})$ be given as in Lemma \ref{maxwell2}. For any $k\geq N+1$, multiplying \eqref{Lichnerowicz2} by $\hat{u}^{k}$ and integrating over $M$, we obtain
   	$$\frac{4(n-1)}{n-2}\int_{M}{\hat{u}^{k}\Delta_{\hat{g}}\hat{u},dv_{\hat{g}}}+\int_{M}{R_{\hat{g}}\hat{u}^{k+1}dv_{\hat{g}}}+\frac{n-1}{n}\int_{M}{\hat{\tau}^{2}\hat{u}^{k+N-1}dv_{\hat{g}}}=\int_{M}{\hat{w}^{2}\hat{u}^{k-N-1}dv_{\hat{g}}}.$$
   	Integration by parts tells us that the first integral is nonnegative, then we get that 
   	\begin{equation*}
   	\begin{aligned}
   	\left(\min{R_{\hat{g}}}\right)\int_{M}{\hat{u}^{k+1}dv_{\hat{g}}}&\leq\int_{M}{\hat{w}^{2}\hat{u}^{k-N-1}dv_{\hat{g}}}\\
   	& \leq\left(\int_{M}{\hat{u}^{k+1}dv_{\hat{g}}}\right)^{\frac{k-N-1}{k+1}}\left(\int_{M}{|\hat{w}|^{\frac{2(k+1)}{N+2}}dv_{\hat{g}}}\right)^{\frac{N+2}{k+1}}\quad\mbox{(by H\"{o}lder inequality)}. 
   	\end{aligned}
   	\end{equation*}
   	It follows that
   	$$\left(\min{R_{\hat{g}}}\right)\left(\int_{M}{\hat{u}^{k+1}dv_{\hat{g}}}\right)^{\frac{N+2}{k+1}}\leq \left(\int_{M}{|\hat{w}|^{\frac{2(k+1)}{N+2}}dv_{\hat{g}}}\right)^{\frac{N+2}{k+1}}.$$
   	Taking $k\to\infty$, we obtain that
   	\begin{equation*}
   	\left(\min{R_{\hat{g}}}\right)\left(\max{\hat{u}}\right)^{N+2}\leq \max{|\hat{w}|}^{2}.
   \end{equation*}
   	Since $\hat{u}=\phi^{-1}u$ and $\hat{w}=\phi^{-N}w$, we get from this inequality that
   	\begin{equation}\label{key2}
   	\left(\min{R_{\hat{g}}}\right)\left(\min{\phi}\right)^{2N}\left(\max{\phi}\right)^{-(N+2)}\left(\max{u}\right)^{N+2}\leq \max{|w|}^{2}.
   	\end{equation}
\end{Remark}
 The following lemma will be used all along the paper.
  \begin{Lemma}(see \cite[Lemma 2.6]{Nguyen})\label{maxpriw}
  Assume that $v,~u$ are respectively a supersolution (resp. subsolution) and a positive solution  to \eqref{Lichnerowicz} associated with a fixed $w$, then 
  $$v\geq u~(\mbox{resp. $\leq$}).$$
  In particular, assume $u_{0}$ (resp. $u_{1}$) is a positive solution to \eqref{Lichnerowicz} associated to $w=w_{0}$ (resp. $w_{1}$). Assume moreover $|w_{0}|\leq |w_{1}|$, then $u_{0}\leq u_{1}$.
  \end{Lemma}
  \begin{proof} 
  We will prove the supersolution case, the remaining cases are similar. Assume that $v,u$ are a supersolution and a positive solution respectively of \eqref{Lichnerowicz} associated to a fixed $w$. Since $u$ is a solution, $u$ is also a subsolution, and hence, as easily checked, so is $tu$ for all constant $t\in (0,1]$. Since $\min v>0$, we now take $t$ small enough s.t. $tu\leq v$. By Proposition \ref{method sub-super}, we then conclude that there exists a solution $u'\in W^{2,p}$ of \eqref{Lichnerowicz} satisfying $tu\leq u'\leq v$. On the other hand, by uniqueness of positive solution of \eqref{Lichnerowicz} given by Theorem \ref{maxwell1}, we obtain that $u=u'$, and hence get the desired conclusion.
  \end{proof}
  \begin{Remark}\label{link_f}
  In the next section, we will study a modified version of \eqref{Lichnerowicz}:
  \begin{equation}\label{current_equation}
  \frac{4(n-1)}{n-2}\Delta u+\left(tR+(1-t)f\right)u+\frac{n-1}{n}\tau^{2}u^{N-1}=\frac{w^{2}}{u^{N+1}},
  \end{equation}
  where $t\in [0,1]$ is a parameter and $f>0$ is a given continuous function. We assume further that $\min{\tau^{2}}>0$. In this situation, Theorem \ref{maxwell1} and Lemma \ref{maxpriw} are still valid for the equation \eqref{current_equation}. For instance, 
 we will see that existence and uniqueness of solutions given in Theorem \ref{maxwell1} is still true here. In fact, suppose that $w\in L^{2p}\setminus \{0\}$. Let $\psi_{f}>0$ be the unique positive solution to 
    \begin{equation}\label{subsolution_given}
    \frac{4(n-1)}{n-2}\Delta u+R_{f}u+\frac{n-1}{n}\tau^{2}u^{N-1}=\frac{w^{2}}{u^{N+1}}
    \end{equation}
  with $R_{f}=\sup_{t}\left(\max\{tR+(1-t)f\}\right)>0$ (here existence and uniqueness of $\psi_{f}$ is proven similarly to Case 1 of Theorem \ref{maxwell1}). It is easy to see that $\psi_{f}$ is a subsolution to \eqref{current_equation}. On the other hand, since $\min{\tau^{2}}>0$, provided that $k>0$ is large enough, $k$ is a supersolution to \eqref{current_equation}, and then the (modified) Lichnerowicz equation \eqref{current_equation} admits a solution by the method of sub-and super-solution (note that $k$ is also a supersolution to \eqref{subsolution_given}, then $k\geq \psi_{f}$ by Lemma \ref{maxpriw}). For any $\phi\in W_{+}^{2,p}$, we now observe that similarly to the proof of Lemma \ref{maxwell2}, $u$ is a solution to \eqref{current_equation} if and only if $\hat{u}=\phi^{-1}u$ is a solution to the following equation
    \begin{equation*}
    \frac{4(n-1)}{n-2}\Delta_{\hat{g}} \hat{u}+\left[R_{\hat{g}}+(1-t)\left(\hat{f}-\hat{R}\right)\right]\hat{u}+\frac{n-1}{n}\hat{\tau}^{2}\hat{u}^{N-1}=\frac{\hat{w}^{2}}{\hat{u}^{N+1}},
    \end{equation*}
    where $\hat{f}=\phi^{-N+2}f$, $\hat{R}=\phi^{-N+2}R$ and $(\hat{g},\hat{w},\hat{\tau})$ is given as in Lemma \ref{maxwell2}. By using this fact, uniqueness of solution to \eqref{current_equation} follows in much the same way as in \cite[Proposition 4.4]{MaxwellRoughCompact}. Similarly, it is not difficult to show that Lemma \ref{maxpriw} remains valid for the (modified) Lichnerowicz equation by the same argument.
  \end{Remark}
  \section{Proof of Theorem \ref{theoremofDHG}}\label{mainsection}
  In this section, we introduce the Leray-Schauder fixed point theorem used in \cite{Nguyen} and obtain another version of the main theorem in \cite{DahlGicquaudHumbert} and \cite{Nguyen}. We first recall the Leray-Schauder fixed point theorem (see e.g. \cite[Theorem 11.6]{GilbargTrudinger}).

  \begin{Theorem} (\textbf{Leray-Schauder fixed point})\label{schaefer}
  Let $X$ be a Banach space and assume that
  $$T:~X\times [0,1]\rightarrow X$$
  is a continuous compact operator, satisfying $T(x,0)=0$ for all $x\in X$. If the set 
  $$K=\left\{x\in X|~~\exists t\in[0,1]~\mbox{such that}~x=T(x,t)\right\}$$
 is bounded, then $T(.,1)$ has a fixed point.
  \end{Theorem}

Before going further, we make the following remark:

\begin{Remark}\label{change_variation}
$(\varphi,W)$ is a solution to the conformal constraint equations w.r.t. the initial data $(g,\tau,\sigma)$ if and only if $\left(C^{-1}\varphi,C^{-\frac{N+2}{2}}W\right)$ is a solution to the conformal constraint equation w.r.t. the initial data $\left(g,C^{\frac{N-2}{2}}\tau,C^{-\frac{N+2}{2}}\sigma\right)$ for any constant $C>0$.
\end{Remark}
\begin{proof}[Proof of Theorem \ref{theoremofDHG}] We divide the proof into three steps\\
 
\textit{\textbf{Step 1.} Construction of a continuous compact operator:}
For any continuous function $f>0$ or $f\equiv R$ if $\mathcal{Y}_{g}>0$, we define the map $T_{f}:~L^{\infty}\times [0,1]\rightarrow L^{\infty}$ as follows. For each $(\varphi,t)\in L^{\infty}\times [0,1]$, there exists a unique $W_{\varphi} \in W^{2,p}$ such that
\begin{equation}\label{vectorequation}
-\frac{1}{2}L^{*}LW_{\varphi}=\frac{n-1}{n}\varphi^{N}d\tau,
\end{equation}
and, by Remark \ref{link_f}, there is a unique $\psi_{\varphi, t} \in W^{2,p}_{+}$ satisfying
$$\frac{4(n-1)}{n-2}\Delta \psi_{\varphi, t}+\left[tR+(1-t)f\right]\psi_{\varphi, t} = -\frac{n-1}{n}t^{2N} \tau^{2} \psi_{\varphi, t}^{N-1} + |\sigma+LW_{\varphi}|^{2}\psi_{\varphi, t}^{-N-1}.$$
We define
$$T_{f}(\varphi,t):= t \psi_{\varphi, t}.$$ 

Following \cite{MaxwellNonCMC} and \cite{DahlGicquaudHumbert}, the mapping $G:~L^{\infty}\rightarrow C^{1}$ defined by $G(\varphi)=W_{\varphi}$, with $W_{\varphi}$ uniquely determined by \eqref{vectorequation}
is continuous and compact. Thus, to show that $T_{f}$ is compact and continuous, it suffices to prove the continuity of
$\hat{T}_{f}:~C^{1}\times [0,1]\rightarrow W^{2,p}_{+}$ defined by $\hat{T}_{f}(W,t)=\psi$, where

\begin{equation}\label{modified.Lichnerowicz}
\frac{4(n-1)}{n-2}\Delta\psi + \left[tR+(1-t)f\right]\psi=-\frac{n-1}{n}t^{2N}\tau^{2}\psi^{(N-1)}+|\sigma+LW|^{2}\psi^{-N-1}.
\end{equation}

We combine the techniques from \cite[Lemma 2.3]{DahlGicquaudHumbert} and \cite[Proposition 3.6]{Nguyen} to prove that
$\hat{T}_{f}$ is continuous. Set $u=\ln{\hat{T}_{f}(W,t)}$. We have from the definition of $\hat{T}_{f}$ that
$$\frac{4(n-1)}{n-2}\left(\Delta u-|d u|^{2}\right)+\left[tR+(1-t)f\right]=-\frac{n-1}{n}t^{2N}\tau^{2}e^{(N-2)u}+|\sigma+LW|^{2}e^{-(N+2)u}.$$

Next, we prove that $\ln\circ\hat{T}_{f}$ is a $C^{1}-$map through the implicit function theorem.
In fact, define $F:C^{1}\times [0,1]\times W^{2,p}\rightarrow L^{p}$ by
$$F(W,t,u)=\frac{4(n-1)}{n-2}\left(\Delta u-|d u|^{2}\right)+\left[tR+(1-t)f\right]+\frac{n-1}{n}t^{2N}\tau^{2}e^{(N-2)u}-|\sigma+LW|^{2}e^{-(N+2)u}.$$
It is clear that $F$ is $C^1$ and, under our assumptions $u = \ln\left(\hat{T}_{f}(W,t)\right)$ is the unique solution to $F\left(W,t,u\right)=0$.
A standard computation shows that the Fr\'{e}chet derivative of $F$ w.r.t. $u$ is given by
  $$F_{u}(W,t)(v)=\frac{4(n-1)}{n-2}\left(\Delta v-\langle du,dv \rangle\right)+\frac{(n-1)(N-2)}{n}t^{2N}\tau^{2}e^{(N-2)u}v+(N+2)|\sigma+LW|^{2}e^{-(N+2)u}v.$$
  We first note that $F_{u}\in C\left(C^{1}\times [0,1], L(W^{2,p},L^{2p})\right)$, where $L(W^{2,p}, L^{2p})$ denotes the Banach space of all linear continuous maps from $W^{2,p}$ into $L^{2p}$. In particular, setting $u_{0}=\ln\left(\hat{T}_{f}(W,t)\right)$ we have
  $$F_{u_{0}}(W,t)(v)=\frac{4(n-1)}{n-2}\left(\Delta v-\langle du_{0},dv \rangle\right)+\left(\frac{(n-1)(N-2)}{n}t^{2N}\tau^{2}e^{(N-2)u_{0}}+(N+2)|\sigma+LW|^{2}e^{-(N+2)u_{0}}\right)v.$$
  Since
  $$\int_{M}{{|\sigma+LW|}^{2}e^{-(N+2)u_{0}}dv}\geq e^{-(N+2)\max |u_{0}|}\int_{M}{{|\sigma+LW|}^{2}dv}=e^{-(N+2)\max |u_{0}|}\left(\int_{M}{{|\sigma|}^{2}dv}+\int_{M}{{|LW|}^{2}dv}\right)>0,$$
  the non-negative term $\left(\frac{(n-1)(N-2)}{n}t^{2N}\tau^{2}e^{(N-2)u_{0}}+(N+2)|\sigma+LW|^{2}e^{-(N+2)u_{0}}\right)$ is not identically $0$.
 Then we can conclude by the maximum principle that $F_{u_{0}}(W,t):~W^{2,p}\rightarrow L^{2p}$ is an isomorphism (see \cite[Theorem 8.14]{GilbargTrudinger}). The implicit function theorem
 then implies that $\ln\circ\hat{T}_{f}$ is a $C^{1}-$function in a neighborhood of $(W,t)$, which proves our claim.\\

 \textit{\textbf{Step 2.} Application of the Leray-Schauder fixed point theorem:} We now set
  $$K=\left\{\varphi\in L^{\infty}|~~\exists
    t\in[0,1]~\mbox{such that}~\varphi=T_{f}(\varphi,t)\right\}.$$
  By the Leray-Schauder fixed point theorem, if $K$ is bounded, then the system \eqref{CE} associated to $(g,\tau,\sigma)$ admits a solution, which is our first assertion.\\

  Assume from now on that $K$ is unbounded. So there exists a sequence $(\varphi_{i}, W_i, t_{i})$ satisfying
  \begin{subequations}\label{modified_CE_with_t}
   \begin{eqnarray}
   \frac{4(n-1)}{n-2}\Delta \varphi_{i}+\left[t_{i}R+(1-t_{i})f\right]\varphi_{i}&=&-\frac{n-1}{n}t_{i}^{2N}\tau^{2}\varphi_{i}^{N-1}+|\sigma+LW_{i}|^{2}\varphi_{i}^{-N-1}\\
   -\frac{1}{2}L^{*}LW_{i}&=&\frac{n-1}{n}t^{N}_{i}\varphi_{i}^{N}d\tau,
   \end{eqnarray}
   \end{subequations}
   with $||\varphi_{i}||_{L^{\infty}}\to +\infty$ (see \cite[Theorem 3.3 or Proposition 3.6]{Nguyen}). We need to discuss the following four possibilities.
   
 \begin{itemize}
 \item\textit{Case 1. (after passing to a subsequence) $t_{i}\to t_{0}>0$:} We argue similarly to \cite[Theorem 1.1]{DahlGicquaudHumbert} or \cite[Theorem 3.3]{Nguyen} to obtain existence of a nontrivial solution $V\in W^{2,p}$ to the limit equation
   $$-\frac{1}{2}L^{*}LV=\sqrt{\frac{n-1}{n}}|LV|\frac{d\tau}{\tau},$$
   which is our second assertion. In fact, we set $\gamma_{i}=||\varphi_{i}||_{\infty}$ and rescale $\varphi_{i},~W_{i}$ and $\sigma$ as follows:
   $$\widetilde{\varphi}_{i}=\gamma_{i}^{-1}\varphi_{i},~~\widetilde{W}_{i}=\gamma_{i}^{-N}W_{i},~~\widetilde{\sigma}_{i}=\gamma_{i}^{-N}\sigma.$$
   Note that by assumption $\gamma_{i}=||\varphi_{i}||_{\infty}\to \infty$ as $i\to\infty$. The system \eqref{modified_CE_with_t} may be rewritten as
   \begin{subequations}\label{rescale}
   \begin{eqnarray}
    \frac{1}{\gamma_{i}^{N-2}}\left[\frac{4(n-1)}{n-2}\Delta\widetilde{\varphi}_{i}+\left(t_{i}R+(1-t_{i})f\right)\widetilde{\varphi}_{i}\right]&=&-\frac{n-1}{n}t_{i}^{2N}\tau^{2}\widetilde{\varphi}_{i}^{N-1}+|\widetilde{\sigma}+L\widetilde{W}_{i}|^{2}\widetilde{\varphi}_{i}^{-N-1}\\
   -\frac{1}{2}L^{*}L\widetilde{W}_{i}&=&\frac{n-1}{n}t^{N}_{i}\widetilde{\varphi}^{N}_{i}d\tau.
   \end{eqnarray}
   \end{subequations}
   Since $||\widetilde{\varphi}_{i}||_{\infty}=1$, we conclude from the vector equation that $\left(\widetilde{W}_{i}\right)_{i}$ is bounded in $W^{2,p}$ and then by the Rellich theorem, (after passing to a subsequence) $\widetilde{W}_{i}$ converges in the $C^{1}$-norm to some $\widetilde{W}_{\infty}$. We now prove that 
   \begin{equation}\label{converging}
   \widetilde{\varphi}_{i}\to\widetilde{\varphi}_{\infty}\coloneqq\left(\sqrt{\frac{n}{n-1}}\frac{|L\widetilde{W}_{\infty}|}{t_{0}^{N}\tau}\right)^{\frac{1}{N}} ~~~\mbox{in $L^{\infty}$.}
   \end{equation}
    Note that if such a statement is proven, passing to the limit in the vector equation, we see that $\widetilde{W}_{\infty}$ is a solution to the limit equation \eqref{limit_non_parameter}. On the other hand, since $||\widetilde{\varphi}_{i}||_{\infty}=1$ for all $i$, we have $||\widetilde{\varphi}_{\infty}||_{\infty}=1$ and, in particular, $\widetilde{W}_{\infty}\nequiv 0$ from \eqref{converging}. Therefore, the non-triviality of $W_{\infty}$ is obtained, and the second assertion follows.\\
    \\
    Given $\epsilon>0$, since $\frac{|L\widetilde{W}_{\infty}|}{\tau}\in C^{0}$, we can choose $\widetilde{\omega}\in C_{+}^{2}$ s.t.
    \begin{equation}\label{omegaandW}
    \biggl|\widetilde{\omega}-\left(\sqrt{\frac{n}{n-1}}\frac{|L\widetilde{W}_{\infty}|}{t_{0}^{N}\tau}\right)^{\frac{1}{N}}\biggr|< \frac{\epsilon}{2}.
    \end{equation}
    To show \eqref{converging}, it suffices to show that
    $$|\widetilde{\varphi}_{i}-\widetilde{\omega}|\leq \frac{\epsilon}{2}$$
    for all $i$ large enough. We argue by contradiction. Assume that it is not true. We first consider the case when (after passing to a subsequence) there exists a sequence $(m_{i})\in M$ s.t. 
    \begin{equation}\label{omegaandpsi}
    \widetilde{\varphi}_{i}(m_{i})>\widetilde{\omega}(m_{i})+\frac{\epsilon}{2}.
    \end{equation}
     By Lemma \ref{maxpriw} and Inequality \eqref{omegaandpsi}, $\widetilde{\omega}+\frac{\epsilon}{2}$ is not a supersolution to the rescaled Lichnerowicz equation. As a consequence, there exists a sequence $(p_{i})\in M$ satisfying
    \begin{align*} \Biggl\{\frac{1}{\gamma_{i}^{N-2}}\left(\frac{4(n-1)}{n-2}\Delta\left(\widetilde{\omega}+\frac{\epsilon}{2}\right)+\left(t_{i}R+(1-t_{i})f\right)\left(\widetilde{\omega}+\frac{\epsilon}{2}\right)\right)+\frac{n-1}{n}t_{i}^{2N}\tau^{2}&\left(\widetilde{\omega}+\frac{\epsilon}{2}\right)^{N-1}\Biggr\}(p_{i})\\
    &<
    \left\{|\widetilde{\sigma}_{i}+L\widetilde{W}_{i}|^{2}\left(\widetilde{\omega}+\frac{\epsilon}{2}\right)^{-N-1}\right\}(p_{i}).
    \end{align*}
    By compactness of $M$, we can assume that $(p_i)$ converges to some $p_\infty \in M$.
    Since $\left(\widetilde{\omega}+\frac{\epsilon}{2}\right)$ and $\tau$ are positive, the previous inequality can be rewritten as follows
    \begin{align*}
    \Biggl\{\frac{n\left(\widetilde{\omega}+\frac{\epsilon}{2}\right)^{N+1}}{(n-1)t_{i}^{2N}\tau^{2}\gamma_{i}^{N-2}}\left(\frac{4(n-1)}{n-2}\Delta\left(\widetilde{\omega}+\frac{\epsilon}{2}\right)+\left(t_{i}R+(1-t_{i})f\right)\left(\widetilde{\omega}+\frac{\epsilon}{2}\right)\right)+&\left(\widetilde{\omega}+\frac{\epsilon}{2}\right)^{2N}\Biggr\}(p_{i})\\
    &<
    \left\{\frac{n}{n-1}|\widetilde{\sigma}_{i}+L\widetilde{W}_{i}|^{2}t_{i}^{-2N}\tau^{-2}\right\}(p_{i}).
    \end{align*}
    Taking $i\to\infty$, due to the facts that $\widetilde{\omega}\in C^{2}_{+}$, $\min{|\tau|}>0$, $t_{i}\to t_{0}>0$, $\gamma_{i}\to\infty$ and $\widetilde{W}_{i}\to\widetilde{W}_{\infty}$ in $C^{1}-$norm, we obtain that
    
   $$\left\{\frac{n\left(\widetilde{\omega}+\frac{\epsilon}{2}\right)^{N+1}}{(n-1)t_{i}^{2N}\tau^{2}\gamma_{i}^{N-2}}\left(\frac{4(n-1)}{n-2}\Delta\left(\widetilde{\omega}+\frac{\epsilon}{2}\right)+\left(t_{i}R+(1-t_{i})f\right)\left(\widetilde{\omega}+\frac{\epsilon}{2}\right)\right)\right\}(p_{i})\to 0,$$
   
     $$\left(\widetilde{\omega}+\frac{\epsilon}{2}\right)^{2N}(p_{i})\to \left(\widetilde{\omega}+\frac{\epsilon}{2}\right)^{2N}(p_{\infty})$$
   and
   $$\frac{n}{n-1}\left(\frac{|\widetilde{\sigma}_{i}+L\widetilde{W}_{i}|}{t_{i}^{N}\tau}\right)^{2}(p_{i})\to \frac{n}{n-1}\left(\frac{|L\widetilde{W}_{\infty}|}{t_{0}^{N}\tau}\right)^{2}(p_{\infty}),$$ 
   
   This proves that 
   $$\widetilde{\omega}(p_{\infty})+\frac{\epsilon}{2}\leq \left(\sqrt{\frac{n}{n-1}}\frac{|L\widetilde{W}_{\infty}|}{t_{0}^{N}\tau}\right)^{\frac{1}{N}}(p_{\infty}),$$
   which contradicts \eqref{omegaandW}.\\
   
   The argument is similar if there exists a sequence $\left(m_{i}\right)\in M$ s.t. $\widetilde{\omega}(m_{i})-\frac{\epsilon}{2}>\widetilde{\varphi}_{i}(m_{i})$.\\
\\
    \item \textit{Case 2. (after passing to a subsequence) $t_{i}\to 0$:} Note that Equations \eqref{modified_CE_with_t} say that the (modified) conformal constraint equations associated to the seed data $(g,t_{i}^{N}\tau,\sigma)$ have a solution $(\varphi_{i},W_{i})$. To derive the last two assertions, we need to free $\tau$ of $t_{i}$ in the seed data. Then, rather than considering $(g,t_{i}^{N}\tau,\sigma)$, by Remark \ref{change_variation}, we can equivalently work on another one more suitable, allowing to remove $t_{i}$ from the mean curvature $\tau$, and hence by straightforward calculations as seen below the sequence $\{t^{n}_{i}\varphi_{i}\}_{i\in\mathbb{N}}$ will naturally appear and play an important role in characterizing our case. In this context, there are three situations arising depending on whether (after passing to subsequence) $t^{n}_{i}||\varphi_{i}||_{L^{\infty}}$ converges to $+\infty$, $0$ or a positive constant. We will address each of them.\\
    \\
    In the first situation, i.e. $ t^{n}_{i}||\varphi_{i}||_{L^{\infty}}\to +\infty$, by Remark \ref{change_variation}, the system \eqref{modified_CE_with_t} may be rewritten  as
   \begin{subequations}\label{modified_CE_infty}
      \begin{eqnarray}
      \label{modified_Lich_infty}
      \frac{4(n-1)}{n-2}\Delta \overline{\varphi}_{i}+\left[t_{i}R+(1-t_{i})f\right]\overline{\varphi}_{i}&=&-\frac{n-1}{n}\tau^{2}\overline{\varphi}_{i}^{N-1}+\left|t_{i}^{\frac{n(N+2)}{2}}\sigma+L\overline{W}_{i}\right|^{2}\overline{\varphi}_{i}^{-N-1}\\
      \label{modified_Vector_infty}
      -\frac{1}{2}L^{*}L\overline{W}_{i}&=&\frac{n-1}{n}\overline{\varphi}_{i}^{N}d\tau,
      \end{eqnarray}
      \end{subequations}
  where $(\overline{\varphi}_{i},\overline{W}_{i})=\left(t_{i}^{n}\varphi_{i},t_{i}^{\frac{n(N+2)}{2}}W_{i}\right)$ and $||\overline{\varphi}_{i}||_{L^{\infty}}=t^{n}_{i}||\varphi_{i}||_{L^{\infty}}\to \infty$. Again, taking $i\to\infty$ we argue similarly to Case 1 and obtain that there exists a nontrivial solution $\overline{W}_{\infty}\in W^{2,p}$ to the limit equation \eqref{limit_non_parameter} as stated in $(ii)$.\\
  \\
  The next situation, i.e. $t_{i}^{n}||\varphi_{i}||_{L^{\infty}}\to 0$, cannot happen. In fact, also by Remark \ref{change_variation} the system \eqref{modified_CE_with_t} may be rewritten as
   \begin{subequations}\label{modified_CE_0}
        \begin{eqnarray}
        \label{modified_Lich_0}
        \frac{4(n-1)}{n-2}\Delta \widehat{\varphi}_{i}+\left[t_{i}R+(1-t_{i})f\right]\widehat{\varphi}_{i}&=&-\frac{n-1}{n}t^{2N}_{i}\gamma_{i}^{N-2}\tau^{2}\widehat{\varphi}_{i}^{N-1}+\left|\gamma_{i}^{-\frac{N+2}{2}}\sigma+L\widehat{W}_{i}\right|^{2}\widehat{\varphi}_{i}^{-N-1}\\
        \label{modified_Vector_0}
        -\frac{1}{2}L^{*}L\widehat{W}_{i}&=&\frac{n-1}{n}t_{i}^{N}\gamma_{i}^{\frac{N-2}{2}}\widehat{\varphi}_{i}^{N}d\tau,
        \end{eqnarray}
  \end{subequations}
  where $\gamma_{i}=||\varphi_{i}||_{L^{\infty}}$ and $(\widehat{\varphi}_{i},\widehat{W}_{i})=(\gamma_{i}^{-1}\varphi_{i},\gamma_{i}^{-\frac{N+2}{2}}W_{i})$. If $f>0$, for any $k\geq N+1$, multiplying \eqref{modified_Lich_0} by $\widehat{\varphi}_{i}^{k}$ and integrating over $M$, we obtain
  \begin{equation}
  \begin{aligned}
  \frac{4(n-1)}{n-2}\int_{M}{\widehat{\varphi}_{i}^{k}\Delta\widehat{\varphi}_{i}dv}+\int_{M}{\left[t_{i}R+(1-t_{i})f\right]\widehat{\varphi}_{i}^{k+1}dv}&+\frac{n-1}{n}\int_{M}{t_{i}^{2N}\gamma_{i}^{N-2}\tau^{2}\widehat{\varphi}_{i}^{k+N-1}dv}\\
  &=\int_{M}{\left|\gamma_{i}^{-\frac{N+2}{2}}\sigma+L\widehat{W}_{i}\right|^{2}\widehat{\varphi}_{i}^{k-N-1}dv}.
  \end{aligned}
  \end{equation}
  Integration by parts tells us that the first integral is nonnegative, then we get that 
  \begin{equation*}
  	\begin{aligned}
  		\left(\min\left\{t_{i}R+(1-t_{i})f\right\}\right)\int_{M}{\widehat{\varphi}_{i}^{k+1}dv}&\leq\int_{M}{\left|\gamma_{i}^{-\frac{N+2}{2}}\sigma+L\widehat{W}_{i}\right|^{2}\widehat{\varphi}_{i}^{k-N-1}dv}\\
  		& \leq\left(\int_{M}{\widehat{\varphi}_{i}^{k+1}dv}\right)^{\frac{k-N-1}{k+1}}\left(\int_{M}{\left|\gamma_{i}^{-\frac{N+2}{2}}\sigma+L\widehat{W}_{i}\right|^{\frac{2(k+1)}{N+2}}dv}\right)^{\frac{N+2}{k+1}}\\
  		&\quad\mbox{(by H\"{o}lder inequality)}. 
  	\end{aligned}
  \end{equation*}
  It follows that
  $$	\left(\min\left\{t_{i}R+(1-t_{i})f\right\}\right)\left(\int_{M}{\widehat{\varphi}_{i}^{k+1}dv}\right)^{\frac{N+2}{k+1}}\leq \left(\int_{M}{\left|\gamma_{i}^{-\frac{N+2}{2}}\sigma+L\widehat{W}_{i}\right|^{\frac{2(k+1)}{N+2}}dv}\right)^{\frac{N+2}{k+1}}.$$
  Taking $k\to\infty$, we obtain that
  \begin{equation}\label{f>0}
  		\left(\min\left\{t_{i}R+(1-t_{i})f\right\}\right)\left(\max{\widehat{\varphi}_{i}}\right)^{N+2}\leq \max\left\{\left|\gamma_{i}^{-\frac{N+2}{2}}\sigma+L\widehat{W}_{i}\right|^{2}\right\}.
  \end{equation}
  However, since $||\widehat{\varphi}_{i}||_{L^{\infty}}=1$ and  $t_{i}^{n}\gamma_{i}\to 0$, we obtain from the vector equation \eqref{modified_Vector_0} that $||L\widehat{W}_{i}||_{L^{\infty}}\to 0$, and then by the fact that $t_{i}\to 0$ and $\gamma_{i}\to+\infty$, taking $i\to+\infty$ we conclude from \eqref{f>0} that $0<\min{f}\leq 0$, which is a contradiction.\\
  \\
  Now if $\mathcal{Y}_{g}>0$ and $f\equiv R$, we let $\hat{g}$ be a conformal metric $\phi^{N-2}g$ where a positive function $\phi\in W^{2,p}_{+}$ is chosen in such a way that $R_{\hat{g}}>0$. Note that $\max{\widehat{\varphi}_{i}}=1$ and that $$t_{i}R+(1-t_{i})f\equiv R \quad\mbox{if $f\equiv R$}.$$ 
  Arguing as to get \eqref{key2} in Remark \ref{changeR}, we then have from \eqref{modified_Lich_0} that 
  \begin{equation}\label{f=R}
  \left(\min{R_{\hat{g}}}\right)\left(\min{\phi}\right)^{2N}\left(\max{\phi}\right)^{-(N+2)}\leq \max\left\{\left|\gamma_{i}^{-\frac{N+2}{2}}\sigma+L\widehat{W}_{i}\right|^{2}\right\}.
  \end{equation}
Taking $i\to+\infty$, since $\gamma_{i}\to +\infty$ and $\left\|L\widehat{W}_{i}\right\|_{L^{\infty}}\to 0$, it follows from \eqref{f=R} that
$$0<\left(\min{R_{\hat{g}}}\right)\left(\min{\phi}\right)^{2N}\left(\max{\phi}\right)^{-(N+2)}\leq 0,$$
which is also a contradiction, and hence the situation where $t_{i}^{n}||\varphi_{i}||_{L^{\infty}}\to 0$ cannot happen as claimed.\\

  For the last one, i.e. $t^{n}_{i}||\varphi_{i}||_{L^{\infty}}\to c$ for some $c>0$, by Remark \ref{change_variation}, we again obtain the system \eqref{modified_CE_infty} where the condition $||\overline{\varphi}_{i}||_{L^{\infty}}\to+\infty$ is replaced by $||\overline{\varphi}_{i}||_{L^{\infty}}\to c$. It follows from \eqref{modified_Vector_infty} that (after passing to a subsequence) $\overline{W}_{i}$ converges to $\overline{W}_{0}$ in $C^{1}$. If $L\overline{W}_{0}\equiv 0$, arguing as to get \eqref{f>0} and \eqref{f=R} in the previous situation, we have from \eqref{modified_Lich_infty} that
 {\small   \begin{equation*}
   \left\{\begin{array}{ll}
   \textrm{if $f>0$:}&\left(\min\left\{t_{i}R+(1-t_{i})f\right\}\right)\left(\max{\overline{\varphi}_{i}}\right)^{N+2}\leq \max\left\{\left|t_{i}^{\frac{n(N+2)}{2}}\sigma+L\overline{W}_{i}\right|^{2}\right\}\to 0,\\
   \textrm{if $\mathcal{Y}_{g}>0$ and if $f\equiv R$:}&\left(\min{R_{\hat{g}}}\right)\left(\min{\phi}\right)^{2N}\left(\max{\phi}\right)^{-(N+2)}\left(\max{\overline{\varphi}_{i}}\right)^{N+2}\leq \max\left\{\left|t_{i}^{\frac{n(N+2)}{2}}\sigma+L\overline{W}_{i}\right|^{2}\right\}\to 0,
   \end{array} \right.
   \end{equation*}}
 where $\hat{g}$ is given as above, i.e., $\hat{g}=\phi^{N-2}g$ with $\phi\in W_{+}^{2,p}$ and $R_{\hat{g}}>0$. This is a contradiction since 
 {\small \begin{equation*}
 \left\{\begin{array}{ll}
 \textrm{if $f>0$:}&\left(\min\left\{t_{i}R+(1-t_{i})f\right\}\right)\left(\max{\overline{\varphi}_{i}}\right)^{N+2}\to\left(\min{f}\right)c^{N+2}>0,\\
 \textrm{if $\mathcal{Y}_{g}>0$ and if $f\equiv R$:}&\left(\min{R_{\hat{g}}}\right)\left(\min{\phi}\right)^{2N}\left(\max{\phi}\right)^{-(N+2)}\left(\max{\overline{\varphi}_{i}}\right)^{N+2}\to\left(\min{R_{\hat{g}}}\right)\left(\min{\phi}\right)^{2N}\left(\max{\phi}\right)^{-(N+2)}c^{N+2}>0.
 \end{array} \right.
 \end{equation*}}
Thus, we obtain $L\overline{W}_{0}\nequiv 0$. Now we can let $\overline{\varphi}_{0}$ be the unique positive solution to the equation
\begin{equation*}
\frac{4(n-1)}{n-2}\Delta \varphi+ f\varphi=-\frac{n-1}{n}\tau^{2}\varphi^{N-1}+|L\overline{W}_{0}|^{2}\varphi^{-N-1}.
\end{equation*}
(Here if $f>0$, existence and uniqueness of $\overline{\varphi}_{0}$ is proven similarly to Case 1 of Theorem \ref{maxwell1}). To show that $(\overline{\varphi}_{0},\overline{W}_{0})$ is a (nontrivial) solution to system \eqref{modified_CE}, which is the first statement of our last assertion, it suffices to show that $\overline{\varphi}_{i}\to \overline{\varphi}_{0}$ in $L^{\infty}$. In fact, since $L\overline{W}_{0}\nequiv0$, arguing similarly to the continuity of $\hat{T}_{f}$ in Step 1, we obtain that the map $\widetilde{T}_{f}~:~U_{\overline{W}_{0}}\times [0,1]\rightarrow W^{2,p}_{+}$ defined by $\widetilde{T}_{f}(w,t)=\varphi$ is continuous, where $U_{\overline{W}_{0}}$ is any given open neighborhood small enough of $|L\overline{W}_{0}|$ in $L^{\infty}$ and $\varphi$ is the unique positive solution to the equation
\begin{equation*}
\frac{4(n-1)}{n-2}\Delta\varphi + \left[tR+(1-t)f\right]\varphi=-\frac{n-1}{n}\tau^{2}\varphi^{N-1}+w^{2}\varphi^{-N-1}.
\end{equation*}
Combining this and the fact that  $\left(t_{i},\left|t_{i}^{\frac{n(N+2)}{2}}\sigma+L\overline{W}_{i}\right|\right)\to \left(0,|L\overline{W}_{0}|\right)$ we obtain $\overline{\varphi}_{i}\to \overline{\varphi}_{0}$ as claimed.
\end{itemize}

To complete our proof, the remaining work is to treat nonuniqueness results for the conformal constraint equations with positive Yamabe invariants.\\

\textit{\textbf{Step 3.} Nonuniqueness of solutions:} Assume that $\mathcal{Y}_{g}>0$. If neither $(i)$ nor $(ii)$ is true, taking $f\equiv R$, arguments above then tell us that there exists a sequence $\{t_{i}\}$ converging to $0$ s.t. the system \eqref{CE} associated to $(g,t^{N}_{i}\tau,\sigma)$ has a solution $(\varphi_{i},W_{i})$ satisfying $||\varphi_{i}||_{L^{\infty}}\to \infty$. On the other hand, we know that provided $\delta>0$ is small enough, the system \eqref{CE} associated to $(g,\delta\tau,\sigma)$ admits a solution $(\varphi_{\delta},W_{\delta})$ such that  $||\varphi_{\delta}||_{L^{\infty}}\leq c_{1}$ for some constant $c_{1}>0$ independent of $\delta$  (see \cite[Theorem 4.8 and Remark 4.9]{Nguyen} or \cite[Theorem 2.1]{GicquaudNgo}). This completes the proof of Theorem \ref{theoremofDHG}.
\end{proof}
\begin{Remark}
If $\mathcal{Y}_{g}<0$, we can omit the assumption $\sigma\nequiv 0$ in Theorem \ref{theoremofDHG}. In fact, let $\{\sigma_{i}\}$ be a sequence of non-zero transverse-traceless tensors converging to 0. Suppose that neither assertion (ii) nor (iii) holds. By Theorem \ref{theoremofDHG}, the system \eqref{CE} associated to $\sigma=\sigma_{i}$ has a solution $(\varphi_{i},W_{i})$. Moreover, these solutions must be uniformly bounded since we assumed that the assertion (ii) is not satisfied. Note that by Case 3 of Theorem \ref{maxwell1} and Lemma \ref{maxpriw} we have that $\varphi_{i}\geq \min\varphi_{0}>0$, where $\varphi_{0}$ is the unique positive solution to the Yamabe equation.
\begin{equation*}
\frac{4(n-1)}{n-2}\Delta \varphi +R\varphi=-\frac{n-1}{n}\tau^{2}\varphi^{N-1}.
\end{equation*}
Thus, taking $i\to\infty$, we obtain our claim.
\end{Remark}

 \section{Applications of Theorem \ref{theoremofDHG}}
  In this section, we show nonexistence and nonuniqueness results and answer a question raised in \cite{MaxwellNonCMC} (see the middle paragraph but one of page 630) as stated in the beginning of this article. For convenience, we will repeat their statements and give the corresponding proofs. We first construct a class of seed data such that the corresponding equations \eqref{CE} and \eqref{limit_non_parameter} have no (non-trivial) solution.
  \begin{Theorem}\label{Theorem.nonexistence_2}(\textbf{\textit{Nonexistence of solution}})
   Let data be given on $M$ as specified in \eqref{condintial} and assume that conditions \eqref{condinitial2} hold. Furthermore, assume that there exists $c>0$ s.t. $\left|L\left(\frac{d\tau}{\tau}\right)\right|\leq 2c\left|\frac{d\tau}{\tau}\right|^{2}$. Let $V$ be a given open neighborhood of the critical set of $\tau$. If $\sigma\nequiv 0$ and $\text{supp}\{\sigma\}\subsetneq M\setminus V$, then both of the conformal constraint equations \eqref{CE} and the limit equation \eqref{limit_non_parameter} associated to the seed data $(g,\tau^{a},\frac{\sigma}{\epsilon a})$ have no solution, provided $a^{-1},\epsilon a>0$ are small enough.
   \end{Theorem}
 
   Examples where the assumptions of this theorem hold are given in \cite{DahlGicquaudHumbert}. Let us sketch briefly their construction.
   Let $M$ be the unit sphere $\mathbb{S}^n$ lying inside $\mathbb{R}^{n+1}$.
   Choose $\tau = \exp(x_1)$ so that $(d\tau/\tau)^\sharp$ is a conformal Killing vector field for the round metric $\Omega$ on $\mathbb{S}^n$. The critical set of $\tau$ then consists of
   the points $(\pm 1, 0, \ldots, 0)$. Let $V$ be an arbitrary neighborhood of these points such that $\mathbb{S}^n \setminus V$ has non-empty interior. By a result of \cite{BCS05}, we can
   deform the metric $\Omega$ on $\mathbb{S}^n \setminus V$ to a new metric $g$ so that $g$ has no conformal Killing vector. The condition $\left|L\left(\frac{d\tau}{\tau}\right)\right|\leq 2c\left|\frac{d\tau}{\tau}\right|^{2}$ is then readily checked. Non-trivial TT-tensors with arbitrarily small support were constructed in \cite{Delay}. His construction shows that there exists $\sigma \not\equiv 0$ whose
   support is contained in $\mathbb{S}^n \setminus V$.
   
   \begin{proof}[Proof of Theorem \ref{Theorem.nonexistence_2}]
   We argue by contradiction. Assume that for each $(a,\epsilon)$ s.t. $a^{-1},\epsilon a>0$ are small enough, there exists $(\varphi_{\epsilon,a},W_{\epsilon,a})$ satisfying the conformal constraint equations
    \begin{subequations}\label{CE_contradiction}
    \begin{eqnarray}
    \frac{4(n-1)}{n-2}\Delta \varphi_{\epsilon,a}+R\varphi_{\epsilon,a}&=&-\frac{n-1}{n}\tau^{2a}\varphi_{\epsilon,a}^{N-1}+\left|\frac{\sigma}{\epsilon a}+LW_{\epsilon,a}\right|^{2}\varphi_{\epsilon,a}^{-N-1},\\
    -\frac{1}{2}L^{*}LW_{\epsilon,a}&=&\frac{n-1}{n}\varphi_{\epsilon,a}^{N}d\tau^{a}.
    \end{eqnarray}
    \end{subequations}
    We will use the rescaling idea of Dahl-Gicquaud-Humbert \cite{DahlGicquaudHumbert} to show that such existence yields a contradiction. In fact, we rescale $\varphi_{\epsilon,a}$, $W_{\epsilon,a}$ as follows
    $$\widetilde{\varphi}_{\epsilon,a}=\epsilon^{\frac{1}{N}}\varphi_{\epsilon,a},~~\widetilde{W}_{\epsilon,a}=\epsilon W_{\epsilon,a}.$$
    The system \eqref{CE_contradiction} may be written as
    \begin{subequations}\label{CE_contradiction_scale}
        \begin{eqnarray}
        \epsilon^{\frac{2}{n}}\widetilde{\varphi}_{\epsilon,a}^{N+1}\left( \frac{4(n-1)}{n-2}\Delta \widetilde{\varphi}_{\epsilon,a}+R\widetilde{\varphi}_{\epsilon,a}\right)&=&-\frac{n-1}{n}\tau^{2a}\widetilde{\varphi}_{\epsilon,a}^{2N}+\left|\frac{\sigma}{ a}+L\widetilde{W}_{\epsilon,a}\right|^{2},\\
        -\frac{1}{2}L^{*}L\widetilde{W}_{\epsilon,a}&=&\frac{n-1}{n}\widetilde{\varphi}_{\epsilon,a}^{N}d\tau^{a}.
        \end{eqnarray}
        \end{subequations}
   We divide our proof into two cases.\\
  
   \textit{\textbf{Case 1.} $\lim_{\epsilon\to 0}{\left\|\widetilde{\varphi}_{\epsilon,a}\right\|_{L^{\infty}}<\infty}$:} Arguing as in the proof of Theorem \ref{theoremofDHG}, taking $\epsilon\to 0$ we obtain that there exists $W_{a}\in W^{2,p}$ satisfying
    \begin{equation}\label{limit_non_parameter_contradiction}
    \begin{aligned}
    -\frac{1}{2}L^{*}LW_{a}&=\sqrt{\frac{n-1}{n}}\left|\frac{\sigma}{a}+LW_{a}\right|\frac{d\tau^{a}}{\tau^{a}}\\
    &=\sqrt{\frac{n-1}{n}}\left|\sigma+aLW_{a}\right|\frac{d\tau}{\tau}.
    \end{aligned}
    \end{equation}
  However, \eqref{limit_non_parameter_contradiction} cannot happen for all $a>0$ large enough by \cite[Proposition 1.6]{DahlGicquaudHumbert}. In fact, take the scalar product of this equation with $d\tau/\tau$ and integrate. It follows that
   \begin{equation}\label{contradiction_assumption}
   \begin{aligned}
   \sqrt{\frac{n-1}{n}}\int_{M}{|\sigma+aLW_{a}|\left|\frac{d\tau}{\tau}\right|^{2}dv}&=-\frac{1}{2}\int_{M}{\langle LW_{a},L(d\tau/\tau)\rangle dv}\\
   &\leq c\int_{M}{\left|\frac{d\tau}{\tau}\right|^{2}|LW_{a}|dv}~~(\mbox{by our assumption}).
   \end{aligned}
   \end{equation}
   Combining this with the fact that $|\sigma+aLW_{a}|\geq a|LW_{a}|-|\sigma|$, we conclude that for $c_{1}=\sqrt{\frac{n}{n-1}}c$
   $$(a-c_{1})\int_{M}{\left|\frac{d\tau}{\tau}\right|^{2}|LW_{a}|dv}\leq \int_{M}{|\sigma|\left|\frac{d\tau}{\tau}\right|^{2}}dv.$$
   Since the right-hand side of the inequality above is bounded, we must have
   \begin{equation}\label{a.to.infty}
   \lim_{a\to\infty}\int_{M}{\left|\frac{d\tau}{\tau}\right|^{2}|LW_{a}|dv}=0.
   \end{equation}
   It then follows from \eqref{contradiction_assumption} that
   $$\lim_{a\to\infty}{\int_{M}{|\sigma+aLW_{a}|\left|\frac{d\tau}{\tau}\right|^{2}dv}}=0.$$
Since $\left|\frac{d\tau}{\tau}\right|\geq\delta$ on $M\setminus V$ for some $\delta>0$ independent of $a$, we then have by the previous inequality that 
   \begin{equation}\label{a.to.infty.2}
   \lim_{a\to\infty}{\int_{M\setminus V}{|\sigma+aLW_{a}|dv}}=0.
   \end{equation}
On the other hand, since $\text{supp}\{\sigma\}\subsetneq M\setminus V$, we get that
$$\left|\int_{M}{\langle\sigma,\sigma+aLW_{a}\rangle dv}\right|\leq ||\sigma||_{L^{\infty}}\int_{M\setminus V}{|\sigma+aLW_{a}| dv}.$$
Together with \eqref{a.to.infty.2}, this shows that
\begin{equation}\label{contradiction_fact}
\lim_{a\to\infty}\int_{M}{\langle\sigma,aLW_{a}\rangle dv}=-\int_{M}{|\sigma|^{2}dv}.
\end{equation}
However, since $\sigma$ is divergence-free, we must have 
$$\int_{M}{\langle\sigma,aLW_{a}\rangle dv}=0$$
for all $a>0$, which contradicts with \eqref{contradiction_fact}.\\

\textit{\textbf{Case 2.} $\lim_{\epsilon\to 0}{\left\|\widetilde{\varphi}_{\epsilon,a}\right\|_{L^{\infty}}}=+\infty$:} 
Set $\gamma_{\epsilon,a}=\left\|\widetilde{\varphi}_{\epsilon,a}\right\|_{L^{\infty}}$, we rescale $\widetilde{\varphi}_{\epsilon,a}$, $\widetilde{W}_{\epsilon,a}$, $\widetilde{\sigma}_{\epsilon,a}$ again
$$ \widehat{\varphi}_{\epsilon,a}=\gamma_{\epsilon,a}^{-1}\widetilde{\varphi}_{\epsilon,a},~~\widehat{W}_{\epsilon,a}=\gamma_{\epsilon,a}^{-N}\widetilde{W}_{\epsilon,a},~~\mbox{and}~~\widehat{\sigma}_{\epsilon,a}=\gamma_{\epsilon,a}^{-N}\sigma.$$
The system \eqref{CE_contradiction_scale} may be rewritten as
\begin{subequations}\label{CE_contradiction_scale_again}
        \begin{eqnarray}
        \epsilon^{\frac{2}{n}}\gamma_{\epsilon,a}^{-(N-2)}\widehat{\varphi}_{\epsilon,a}^{N+1}\left( \frac{4(n-1)}{n-2}\Delta \widehat{\varphi}_{\epsilon,a}+R\widehat{\varphi}_{\epsilon,a}\right)&=&-\frac{n-1}{n}\tau^{2a}\widehat{\varphi}_{\epsilon,a}^{2N}+\left|\frac{\widehat{\sigma}}{ a}+L\widehat{W}_{\epsilon,a}\right|^{2},\\
        -\frac{1}{2}L^{*}L\widehat{W}_{\epsilon,a}&=&\frac{n-1}{n}\widehat{\varphi}_{\epsilon,a}^{N}d\tau^{a}.
        \end{eqnarray}
        \end{subequations}

Arguing as in the proof of Theorem \ref{theoremofDHG}, and taking $\epsilon\to 0$ we again obtain that there exists a nontrivial solution $W_{a}\in W^{2,p}$ satisfying the limit equation
\begin{equation}\label{limit_equation_contradiction_2}
-\frac{1}{2}L^{*}LW_{a}=\sqrt{\frac{n-1}{n}}\left|LW_{a}\right|\frac{d\tau^{a}}{\tau^{a}}=a\sqrt{\frac{n-1}{n}}\left |LW_{a}\right|\frac{d\tau}{\tau}.\\
\end{equation}
 Our treatment for such limit equation is also similar to the previous case. In fact, take the scalar product of this equation with $d\tau/\tau$ and integrate. It follows that
 \begin{equation}
   \begin{aligned}
   a\sqrt{\frac{n-1}{n}}\int_{M}{|LW_{a}|\left|\frac{d\tau}{\tau}\right|^{2}dv}&=-\frac{1}{2}\int_{M}{\langle LW_{a},L(d\tau/\tau)\rangle dv}\\
   &\leq c\int_{M}{|LW_{a}|\left|\frac{d\tau}{\tau}\right|^{2}dv}~~(\mbox{by our assumption}).
   \end{aligned}
   \end{equation}
 Then assuming $a>\sqrt{\frac{n}{n-1}}c$, we obtain that $\int_{M}{|LW_{a}|\left|\frac{d\tau}{\tau}\right|^{2}dv}=0$, and hence $|LW_{a}|\left|\frac{d\tau}{\tau}\right|\equiv 0$. Thus, we obtain from \eqref{limit_equation_contradiction_2} that $W_{a}\equiv 0$, provided that $(M,g)$ has no conformal Killing vector field. This is a contradiction with the fact that $W_{a}$ is nontrivial.\\
 \\
Since Case 2 coincides with the situation of nonexistence of a solution to the limit equation \eqref{limit_non_parameter}, the proof is completed.
\end{proof}
As direct consequences of Theorem \ref{theoremofDHG} and \ref{Theorem.nonexistence_2}, we have the following results. 
\begin{Corollary} \label{Maxwell'squestion_2}(\textbf{\textit{An answer to Maxwell's question}})
Let $(M,g,\tau)$ be given as in Theorem \ref{Theorem.nonexistence_2}. If $\mathcal{Y}_{g}>0$, then the conformal constraint equations \eqref{CE} associated to $(g,\tau^{a},0)$ have a (nontrivial)  solution for all $a>0$ large enough.
 \end{Corollary}
 \begin{proof}
We have by Theorem \ref{Theorem.nonexistence_2} that for all $a^{-1}, \epsilon a>0$ small enough, seed  data $(g,\tau^{a},\frac{\sigma}{\epsilon a})$ satisfies neither $(i)$ nor $(ii)$ in Theorem \ref{theoremofDHG}, provided $\sigma$ is given as in Theorem \ref{Theorem.nonexistence_2}. Thus, our corollary is proven by the first statement in the assertion $(iii)$ of Theorem \ref{theoremofDHG} with $f\equiv R$. The proof is completed.
 \end{proof}
 \begin{Corollary}(\textbf{\textit{Nonuniqueness of solutions}})\label{theoremnonuniaueness}
 Assume that $(M,g,\tau,\sigma,a,\epsilon)$ is given as in Theorem \ref{Theorem.nonexistence_2}. If $\mathcal{Y}_{g}>0$, then there exists a sequence $\{t_{i}\}$ converging to $0$ s.t. the conformal constraint equations \eqref{CE} associated to $(g,t_{i}\tau^{a},\frac{\sigma}{\epsilon a})$ have at least two solutions.
 \end{Corollary} 
 \begin{proof}
 The same arguments as in Corollary \ref{Maxwell'squestion_2} works here.  More precisely, the only difference from the previous corollary is that we will use the second conclusion in the assertion $(iii)$ of Theorem \ref{theoremofDHG} with $f\equiv R$ instead of the first, and then the corollary follows.
 \end{proof}
\renewcommand{\bibname}{References}

\end{document}